\documentclass[11pt]{amsart}

\usepackage{amsmath,amsthm,amsfonts,amssymb,mathrsfs}
\usepackage{inputenc,mathrsfs}

\numberwithin{equation}{section}
\usepackage[dvips]{graphicx}
\usepackage[colorlinks]{hyperref}

\newtheorem{definition}{Definition}[section]
\newtheorem{theorem}{Theorem}[section]

\newtheorem{lemma}{Lemma}[section]
\newtheorem{corollary}{Corollary}[section]

\theoremstyle{remark}
\newtheorem{remark}{Remark}[section]
\newtheorem*{notation}{Notation}
\newtheorem*{convention}{Conventions}
\newtheorem*{acknowledgements}{Acknowledgements}
\newtheorem*{MSC}{Mathematics subject clasification (2010)}

\title[Boundary estimates for the Ricci flow.]{Boundary estimates for the Ricci flow.}

\author[]{Panagiotis Gianniotis}
\address{Department of Mathematics\\ 
University College London\\
25 Gordon St, London WC1E 6BT, United Kingdom}
\email{p.gianniotis@ucl.ac.uk}

\DeclareMathOperator{\dist}{\textrm{dist}}

\DeclareMathOperator{\Rm}{\textrm{Rm}}
\DeclareMathOperator{\Ric}{\textrm{Ric}}
\DeclareMathOperator{\inj}{\textrm{inj}}
\DeclareMathOperator{\expo}{\textrm{exp}}
\DeclareMathOperator{\ide}{\textrm{id}}

\begin{document}
\maketitle
\begin{abstract}
In this paper we consider the Ricci flow on manifolds with boundary with appropriate control on its  mean curvature and conformal class. We obtain higher order estimates for the curvature and second fundamental form near the boundary, similar to Shi's local derivative estimates. As an application, we prove a version of Hamilton's compactness theorem  in which the limit has boundary.

Finally,  we show that in dimension three the second fundamental form of the boundary and its derivatives are a priori controlled in terms of the ambient curvature and some non-collapsing assumptions. In particular, the flow exists as long as the curvature remains bounded, in contrast to the general case where control on the second fundamental form is also required.

\begin{MSC}
 Primary: 53C44;  Secondary: 35K51 
\end{MSC}
\end{abstract}
\section{Introduction.}
Let $g(t)$ be Ricci flow on a manifold $M$, namely a solution to
$$\frac{\partial}{\partial t} g(t)=-2\Ric(g(t)),$$
not necessarily complete. Shi in \cite{Shi1} obtained local a priori estimates for derivatives of the curvature along Ricci flow, depending only on a curvature bound. See also  Theorem 13.1 in \cite{hamilton_formation}. The following theorem states a global version of these estimates.
\begin{theorem}
Let $(M^n,g(t))$ be a Ricci flow on a closed manifold $M$, $t\in [0,\frac{1}{K}],$ and assume 
$$|\Rm(g(t))|_{g(t)}\leq K.$$
Then, for every positive integer $j$ there exists $C_j=C(n,j)>0$ such that
$$|\nabla^j \Rm(g(t))|_{g(t)}\leq \frac{C_j K}{t^{j/2}},$$
in $M$ for all $t\in (0,\frac{1}{K}]$.
\end{theorem}

Such estimates not only reveal the smoothing character of Ricci flow, but they are also an essential ingredient of a compactness theorem for sequences of Ricci flows, proven by Hamilton in \cite{hamcomp}. This theorem allows the blow-up analysis of singularities, and is an important tool in the study of the global behaviour of the flow.

When $M$ is a manifold with boundary, although there have been several local existence results for the Ricci flow, very little is known regarding its global behaviour, even in dimension 3. 

Regarding local existence, Shen  in \cite{YShen} and Pulemotov in \cite{Pulemotov} consider natural Neumann-type boundary  conditions.  On the other hand, in \cite{gianniotis} the author considers a mixed Dirichlet-Neumann  boundary value problem for the Ricci flow, motivated by work of Anderson on boundary value problems for Einstein metrics in \cite{Anderson1}, where the conformal class of the boundary and its mean curvature provide an elliptic boundary value problem for the Einstein equations.

On the global behaviour, Shen in \cite{YShen} studies the case the initial metric has positive Ricci curvature and the boundary remains totally geodesic, while Cortissoz in \cite{Cortissoz} studies the same when the boundary is convex and umbilic. Moreover, with  Murcia in \cite{Cort} they consider  the  2-dimensional  Ricci flow assuming positive Gauss curvature and convex boundary.

To understand interior singularities, where after rescaling the boundary is sent to infinity, a local version of Shi's estimates and Hamilton's compactness theorem would suffice. However, when singularities form close to the boundary one needs a compactness theorem which can handle the possibility that the limit is a manifold with boundary. Such compactness result would require some analogue of Shi's estimates for the curvature to be valid on a neighbourhood of the boundary, together with higher order estimates of its second fundamental form.

In \cite{Cortissoz,cort_rotsym,Cort} the authors obtain estimates for derivatives of the curvature valid for the boundary value problems considered. However, it is not clear whether these techniques can be adapted to deal with boundary value problems like that in \cite{gianniotis}. 

The main difficulty is essentially that the curvature doesn't seem to satisfy boundary conditions which would allow an application of the maximum principle. Moreover, as mentioned above, the possible applications require control on higher derivatives of the second fundamental form as well.

In this paper, with a technique inpired by the work of Anderson in \cite{and90}, we obtain both these higher order estimates at the same time for the boundary value problem introduced in \cite{gianniotis}. Moreover, it seems that our approach could provide similar estimates for other boundary value problems for the Ricci flow as well. 

To summarize the result in \cite{gianniotis}, let $(M,g_0)$ be a compact Riemannian manifold with boundary,  $\gamma(t)$ be any smooth one parameter family of metrics on $\partial M$ and $\eta(t)$ be any smooth function on $\partial M\times [0,\infty)$, satisfying certain zeroth order compatibility conditions. Then, there exists a Ricci flow $g(t)$ on $M$, smooth for $t>0$, satisfying
\begin{eqnarray}
[g^T(t)]&=&[\gamma(t)],\nonumber\\
\mathcal H(g(t))&=&\eta(t).\nonumber
\end{eqnarray}
Here, $g^T$ denotes the induced metric on $\partial M$, $[\;\cdot\;]$ the conformal class, and $\mathcal H(g)$ the mean curvature of $\partial M$ with respect to $g$. Morevover, $g(t)\rightarrow g_0$ in the $C^{1,\alpha}$ Cheeger-Gromov sense.

In \cite{and90} Anderson uses the ellipticity of the Ricci tensor in harmonic coordinates and elliptic regularity to obtain a $C^{1,\alpha}$ compactness result for the class of manifolds with two sided Ricci curvature bounds and an injectivity radius bound. In particular, he introduces the notion of harmonic radius, the maximal radius of geodesic balls on which there exist harmonic coordinates such that the scale-invariant $C^{1,\alpha}$ norms of the metric components are uniformly bounded. The compactness result is obtained by showing, via a blow up argument, that the Ricci curvature and injectivity radius bounds control the harmonic radius from below. This scheme was later used  in \cite{AndTaylor} and \cite{knox_compactness} to obtain $C^{1,\alpha}$- compactness results for manifolds with boundary.

The primary purpose of this paper, and the content of Section \ref{bndryestimates},  is to adapt this scheme to the parabolic setting, the Ricci flow, in order to estimate derivatives of the curvature near the boundary and the second fundamental form of the boundary. The motivation comes from the fact that the equivalent Ricci-DeTurck flow satisfies a parabolic boundary value problem for which one has the parabolic regularity estimates of Solonnikov from \cite{MR0211083}. 

We briefly describe the main result. Let $M$ be a manifold with compact boundary. In the following, $i_{b,g(0)}$ denotes the size of a collar neighbourhood of the boundary diffeomorphic to $\partial M\times [0,i_{b,g(0)})$ via the normal exponential map of $\partial M$, as described in Definition \ref{inj2}. 

A Ricci flow with $\Lambda$-controlled boundary is, briefly, a Ricci flow with a choice $\gamma(t)$ of representatives of $[g^T(t)]$, uniformly equivalent to $g^T(t)$. Moreover, in $\gamma$-harmonic coordinates, $\gamma(t)$ and the mean curvature $\mathcal H(g(t))$ are controlled in the H\"older sense of order $m+\epsilon$ and $m-1+\epsilon$ respectively, where $m$ is a large  integer.  See Definition \ref{controlled_data} for more details.

\begin{theorem}\label{main_result}
Let $(M,g(t),\gamma(t))$, $t\in [0,T]$,  be a complete Ricci flow with $\Lambda$-controlled boundary in $(0,T]$. Suppose
\begin{enumerate}
\item $|\Rm(g(t))|_{g(t)}\leq K$ in $M$ and $|\mathcal A(g(t))|_{g^T(t)}\leq K$ on $\partial M$ for all $t$.
\item $i_{b,g(0)}\geq i_0$.
%\item $\inj_{g^T(0)}\geq i_0$.
\end{enumerate}
For any $j=1,\ldots,m-2$ and $\tau>0$, there exists a constant $C=C(n,\tau,T,\Lambda,l,j,K,i_0)>0$ such that for all $t\in [\tau,T]$ 
\begin{eqnarray}
|\nabla^j \Rm(g(t))|_{g(t)}&\leq& C \quad \textrm{in $M$},\nonumber\\
|\nabla^{j+1} \mathcal A(g(t))|_{g^T(t)}&\leq& C, \quad \textrm{on $\partial M$}.\nonumber 
\end{eqnarray}
\end{theorem}

The strategy is to define an analogous notion of ``harmonic coordinates'' suitable for our purposes, using a local harmonic map heat flow. Then, parabolic theory provides the higher order estimates near the boundary which allow us to obtain a compactness result under the assumption that the ``parabolic radius", the analogue of the harmonic radius, is bounded below. Finally, a blow up argument shows that the parabolic radius is controlled in terms of geometric data, thus obtaining the higher order estimates of Theorem \ref{main_result}. Moreover, in Theorem \ref{estimates} we prove a local version of the estimate above.

In Section \ref{comp}, we apply these estimates to obtain a version of Hamilton's compactness theorem for sequences of Ricci flows on manifolds with boundary, in Theorem \ref{comp2}. Then, in Section \ref{threemnflds}, we show that in dimension three it is possible to obtain a priori control of the second fundamental form of the boundary and its derivatives along the Ricci flow, in terms of a curvature bound and some non-collapsing assumptions. This is Theorem \ref{sffb}. In particular, this allows us in Corollary \ref{cp} to improve the continuation principle proven in \cite{gianniotis}, in that only a curvature bound suffices for the continuation of the flow. Thus, we rule out the possibility that the flow will develop singularities where the second fundamental form blows up in finite time while the curvature remains bounded.

\section{Preliminary definitions.}
 Let $M$ be a $(n+1)$-dimensional manifold with non-empty boundary $\partial M$.  Given $x\in \partial M$, $(M,x)$ will be called a pointed manifold with boundary. Let $h$ be a, possibly incomplete (in the sense of metric spaces), Riemannian metric on $M$. We will denote by $\overline M$ and $\overline{\partial M}$ the metric completions of $M$ and $\partial M$ respectively. $(M,h)$ is said to be complete, if $\overline M=M$.

\begin{definition}\label{inj1}
\mbox{}
Let $x\in \partial M$.
\begin{enumerate}
\item  $i_{b,loc,h}(x)$ will denote the maximal number with the following properties. We require that the metric ball $B^\partial_h(x,i_{b,loc,h}(x))\subset \partial M$ has compact closure and that the normal exponential map over this ball induces a diffeomorphism between $B^\partial_h(x, i_{b,loc,h}(x))\times [0,i_{b,loc,h}(x))$ and its image. 
\item $\inj_{\partial M,h^T}(x)$ will be the maximal radius such that the $h^T$-exponential map on $B(0,\inj_{\partial M,h^T}(x))\subset \mathbb R^n$ is a diffeomorphism onto its image. Similarly, we denote by $\inj_{M,h}(x)$ the interior injectivity radius defined at an interior point $x$.
\end{enumerate}
\end{definition}

We will also need the concept of the boundary injectivity radius, defined below for complete manifolds with boundary (see also \cite{AndTaylor}).

\begin{definition}\label{inj2}
Let $(M,h)$ be a complete Riemannian manifold with boundary $\partial M$. The boundary injectivity radius $i_{b,h}$ is the maximal number with the property that the normal exponential map of $\partial M$ is a diffeomorphism between $\partial M\times [0,i_{b,h})$ and $\{ x| \dist_h(x,\partial M)< i_{b,h}\}$.
\end{definition}

Now, let $g(t)$ be a one parameter family of uniformly equivalent metrics on $M$, $t\in (a,b]$ and $0\in (a,b]$. Set $\Sigma=\overline{\partial M}\setminus \partial M$ and define 
$$D_g(x,t)=\min\left\{\dist_{g^T(t)}(x,\Sigma), (t-a)^{1/2},1\right\}.$$
If $\Sigma=\emptyset$ we set $D_g(x,t)=1$.

Write $B(0,r)$ and $B(0,r)^+$ for the Euclidean open ball in $\mathbb R^n$ and half ball in the upper half space $\mathbb R^{n+1}_+=\{x^0\geq 0\}$ respectively, of radius $r$.

We are going to need  a scale invariant version of the parabolic $C^{l,l/2}$ norm on $B(0,r)\times [0,r^2]$ and $B(0,r)^+\times [0,r^2]$. In the following definitions, we will denote both parabolic domains by $G$, for simplicity. First, for any $\mu\in(0,1)$ we introduce the notation
\begin{eqnarray}
\langle v \rangle_{\mu,x}&=&\sup_{(x,t),(y,t)\in G}\frac{|v(x,t)-v(y,t)|}{|x-y|^\mu},\nonumber\\
\langle v \rangle_{\mu,t}&=&\sup_{(x,t),(x,t')\in G}\frac{|v(x,t)-v(x,t')|}{|t-t'|^\mu}.\nonumber
\end{eqnarray}
Now, fix $p>n+3$, $\epsilon=1-\frac{n+3}{2}$ and a large positive integer $m$. For $l=m+\epsilon$, the scale invariant $C^{l,l/2}$ norm on $G$ is defined by
\begin{eqnarray}
| v |_{l,r}^* &=& \sum_{k=0}^{m}  \sum_{2q+|s|=k} r^k|\partial_t^q\partial_x^s v|_{C^0}+ \sum_{2q+|s|=m}r^l\langle \partial_t^q\partial_x^s v \rangle_{\epsilon,x}+ \sum_{0<l-2q-|s|<2}r^l\langle \partial_t^q\partial_x^s v \rangle_{\frac{l-2q-|s|}{2},t}.\nonumber
\end{eqnarray}
We will also use the corresponding norms defined in smaller time intervals, namely in the domains $B(0,r)\times [(\eta r)^2,r^2]$ and $B(0,r)^+\times [(\eta r)^2,r^2]$ which will be denoted by $|u|_{l,r,\eta}^*$, for $0\leq \eta <1$.

\begin{convention}
In the following, if the order of some convergence is higher than 3 it will be a function of $m$ and will be referred to as smooth throughout the paper. Also $Q>1$  will be a fixed constant and $a<0\leq b$.
\end{convention}

\section{Boundary estimates.}\label{bndryestimates}
In this section we prove derivative estimates for the curvature near the boundary, and the second fundamental form of the boundary. Such estimates are essentially a consequence of parabolic regularity, given high order control of appropriate boundary data. 

Here we will consider the boundary data in \cite{gianniotis}, namely the mean curvature and the conformal class of the boundary. In particular, we begin with the following definition .

\begin{definition}\label{controlled_data} 
A Ricci flow on a manifold with boundary $(M,g(t))$, $t\in (a,b]$, has boundary with $\Lambda$-controlled conformal class and mean curvature in the interval $(a,b]$, if there is a smooth one-parameter family $\gamma(t)$ of metrics on $\partial M$, such that 
\begin{enumerate}
\item $[g^T(t)]=[\gamma(t)]$ and $\Lambda^{-2}\gamma(t)\leq g^T(t) \leq \Lambda^2\gamma(t)$ for all $t\in(a,b]$.
\item For every $(\bar x,\bar t)\in \partial M\times (a,b]$, set $\tilde \gamma(s)=\gamma(s+\bar t-r^2)$ and $\tilde{\mathcal H}(s)=\mathcal H(g(s+\bar t-r^2))$, where  $\mathcal H(g)$ denotes the mean curvature of the boundary with respect to the metric $g$. 
We require that for any $r\leq \rho_{\Lambda}(\bar x,\bar t)=\Lambda^{-1} D_\gamma(\bar x,\bar t)$ there exist $\gamma(\bar t)$-harmonic coordinates $u:U\rightarrow B(0,r)$ around $\bar x$ such that
\begin{enumerate}
\item $Q^{-1}\delta\leq \tilde\gamma(s) \leq Q\delta$, \qquad \textrm{in $B(0,r)$ and $s\in[0,r^2]$,}
\item $|\tilde \gamma_{\alpha\beta}|_{l,r}^*\leq Q$, where $\alpha,\beta=1,...,n$,
\item $|\tilde{\mathcal H}|_{l-1,r}^*\leq Q$,
\end{enumerate}
\end{enumerate}
Here $\delta$ denotes the Euclidean metric. Such triplet $(M,g(t),\gamma(t))$ will be called a Ricci flow with $\Lambda$-controlled boundary in $(a,b]$.

Moreover, we will say that $(M,g(t),\gamma(t))$, $t\in [a,b]$, has $\Lambda$-controlled boundary in $[a,b]$ if it is $\Lambda$-controlled in $(a,b]$ and in addition for all $\bar x\in \partial M$ and all $r\leq \Lambda^{-1}\min\{\dist_{g^T(0)}(\bar x, \overline{\partial M}\setminus \partial M),1\}$ there exist $\gamma(0)$-harmonic coordinates around $\bar x$ in $\partial M$ in which
\begin{eqnarray}
|\gamma_{\alpha\beta}|^*_{l,r}&\leq& Q,\nonumber\\
|\mathcal H(g(\cdot))|^*_{l-1,r}&\leq& Q.\nonumber
\end{eqnarray}
\end{definition}

A few comments are needed in order to clarify the concept of a Ricci flow with $\Lambda$-controlled boundary. First of all, since we are going to be considering sequences of such Ricci flows and their limits, an issue that might arise is that information for the conformal class of the boundary is lost in the limit. To illustate how this may happen,  consider the sequence of Riemannian manifolds  $(S^n,g_k)$, and let  $g_k=kg_{round}$, $k=1,2,...$, $\gamma_k=g_{round}$. Unfortunately, although $[g_k]=[\gamma_k]$ for every $k$, this does not hold in the limit. To see this, observe that $(S^n,g_k,p)$  converges to  $(\mathbb R^n, g_{Euclid},0)$ in the pointed Cheeger-Gromov topology, whereas $(S^n,\gamma_k,p)$ converges to $(S^n,g_{round})$. An assumption of the form of item (1) of the definition above prevents such phenomena to occur.

Item (2) of the definition may be seen as a parabolic analogue of a lower bound on the harmonic radius (as in \cite{and90}) for $\gamma(t)$. At the same time, we require control of the mean curvature of $g$ with respect to the geometry induced by $\gamma$.

\subsection{Parabolic coordinates.} 
\begin{definition}
Let $(M,g(t),\gamma(t))$ be a Ricci flow with $\Lambda$-controlled boundary and $(\bar x,\bar t)\in \partial M\times (a,b]$. Parabolic coordinates of radius $r$ around $(\bar x,\bar t)$ on $(M,g(t),\gamma(t))$ consist of a pair $(\Omega_s,\phi_s)$, where
\begin{enumerate}
\item For $s\in [0,r^2]$, $\phi_s:\Omega_s\rightarrow B(0,r)^+$ are coordinates on $M$ around $\bar x$ such that
\begin{eqnarray}
\frac{\partial}{\partial s}\phi_s&=&\Delta_{g(\bar t-r^2+s),\delta}\phi_s,\nonumber\\
\phi_s|_{\Omega_s\cap\partial M}&=&\phi_0|_{\Omega_0\cap\partial M},\nonumber
\end{eqnarray}
\item $\phi_s|_{\Omega_s\cap\partial M}: \Omega_s\cap\partial M\rightarrow B(0,r)$ form $\gamma(\bar t)$-harmonic coordinates on $\partial M$ around $\bar x$.
\end{enumerate}
\end{definition} 

To motivate this definition, notice that the push-forward flows $\hat g(s)=(\phi_s)_* g(\bar t-r^2+s)$ evolve by the Ricci-DeTurck flow, which is a parabolic equation. Namely,
$$\frac{\partial}{\partial s}\hat g=-2\Ric(\hat g)+\mathcal L_{\mathcal W}\hat g,$$
where $\mathcal W^i=\hat g^{pq}\widehat \Gamma^i_{pq}$, see for example \cite{MR2274812}. This may be viewed as the parabolic analogue of the fact that the Ricci tensor becomes an elliptic operator in harmonic coordinates. 

\begin{lemma}[Existence of parabolic coordinates]
Given $(M,g(t),\gamma(t))$ and $(\bar x,\bar t)\in \partial M\times (a,b]$, there exists a small $r>0$ and parabolic coordinates $(\Omega_s,\phi_s)$ of radius $r$ based at  $(\bar x,\bar t)$ such that $\hat g(s):=(\phi_s)_*g(\bar t -r^2+ s)$ satisfies the bounds
\begin{eqnarray}
Q^{-1}\delta\leq \hat g(s)\leq Q\delta &&\qquad \textrm{ for $s\in [0,r^2]$,}\label{prad1}\\
|\hat g_{ij}|^*_{\epsilon,r}\leq Q,\label{prad2}\\
\sup_{\eta\in (0,1)} \eta^8 |\hat g_{ij}|^*_{2+\epsilon,r,\eta}\leq Q.&&\label{prad3}
\end{eqnarray}
\end{lemma}
\begin{proof}
Consider an arbitrary coordinate system $u:U\rightarrow B(0,2r')^+$ around $\bar x$ which restricts to $\gamma(\bar t)$-harmonic coordinates on $\partial M$ and satisfies $g(\bar t)_{ij}|_{u=0}=\delta_{ij}$. Clearly, for small $r'$ the condition $Q^{-1}\delta\leq g(\bar t)\leq Q \delta$ holds in  $B(0,2r')^+$. In the following, we will use $u$ to identify $U$ with $B(0,2r')^+$.

Let $B$ be a domain with smooth boundary such that $B(0,\frac{4r'}{3})^+\subset B\subset B(0,\frac{5r'}{3})^+$.   For appropriate $r>0$ which we will define later, consider solutions $\varphi_s:(B,g(\bar t-r^2+s))\rightarrow (B,\delta)$ of the Dirichlet problem for the harmonic map heat flow 
\begin{eqnarray}
\frac{\partial}{\partial s}\varphi_s&=&\Delta_{g(\bar t-r^2+s),\delta}\varphi_s,\nonumber\\
\varphi_0&=&\ide_B,\nonumber\\
\varphi_s|_{\partial B}&=&\ide_{\partial B},\quad\textrm{for every}\;s\in[0,r^2].\nonumber
\end{eqnarray}

Since the metric on the target of $\varphi_s$ is the Euclidean, the flow simplifies to the following initial-boundary value problem for a system of linear heat equations in $B$.
\begin{eqnarray}
\frac{\partial}{\partial s}\varphi_s^i&=&g_s^{kl}\partial^2_{kl}(\varphi_s^i)-g_s^{kl}\Gamma_{s,kl}^m\partial_m(\varphi_s^i),\nonumber\\
\varphi_0^i&=&u^i,\nonumber\\
\varphi_s^i|_{\partial B}&=&u^i|_{\partial B},\quad\textrm{for every}\;s\in[0,r^2],\nonumber
\end{eqnarray}
where $u^i$ are the coordinates in $\mathbb R^{n+1}$, $\varphi_s=(\varphi_s^0,\ldots,\varphi_s^n)$ and we write $g_s, \Gamma_s$ for $g(\bar t-r^2+s)$ and its Christoffel symbols.

Now,  $\varphi_s^i-u^i$ satisfy linear heat equations with homogeneous initial and boundary conditions. It follows by the parabolic estimates in \cite{MR0241822} that there is a $C$ depending on $g_{ij}(\cdot)$ such that
$$||\varphi_s^i||_{W^{2,1}_p(B\times (0,r^2))}\leq C,$$
where $$|| u ||_{W^{2,1}_p( B\times (0,r^2))}=\left( \int_{B\times (0,r^2)} (|u|^p  +|Du|^p + |D^2 u|^p + |\partial_t u|^p )dx dt  \right)^{1/p}.$$
Hence, by the embedding  $ W^{2,1}_p\subset C^{1+\epsilon, \frac{1+\epsilon}{2}}$, for $p>n+3$ and $\epsilon=1-\frac{n+3}{p}$ (see \cite{MR0241822}  for instance), there exists a $\tau>0$ such that $\varphi_s$ are diffeomorphisms for $s\in [0,\tau]$. In particular $\tau$ does not depend on $r$ or $\bar t$ for any given $(g(\cdot),\gamma(\cdot))$. Moreover a uniform $C^{\epsilon,\epsilon/2}$-bound of $\partial \varphi_s$ in $B\times [0,\tau]$ holds.

To obtain higher order control of $\varphi_s$ consider the map $v=(v^0,\ldots,v^n)$ where $v^i=u^i+s(\varphi^i_s-u^i)$. It satisfies the system
\begin{eqnarray}
\frac{\partial}{\partial s}v^i-\Delta_{g_s}v^i=(\varphi_s^i-u^i)-(1-s)\Delta_{g_s} u^i\nonumber
\end{eqnarray}
and the boundary condition
\begin{eqnarray}
v^i=u^i,\quad \textrm{on}\; \partial B\times [0,r^2].\nonumber
\end{eqnarray}
Observe that the first order compatibility condition holds, namely
$$\left.\frac{\partial}{\partial s}v^i\right|_{\partial B\times 0}=\Delta_g (v^i-u^i)|_{\partial B\times 0}=0=\left.\frac{\partial}{\partial s}u^i\right|_{\partial B\times 0}.$$

Thus, by parabolic regularity, $v^i$ belong (and are uniformly bounded) in $C^{3+\epsilon,\frac{3+\epsilon}{2}}(B\times[0,\tau])$. This shows that $\varphi_s$ is uniformly bounded in $C^{3+\epsilon,\frac{3+\epsilon}{2}}(B\times[\tau',\tau])$ for any $0<\tau'<\tau$. Similarly, one can show that $\varphi_s$ is in fact smooth for $s\geq\tau'$.

Now, let $r \leq \min(r',\tau^{1/2})$ . Since $\partial_k \varphi^i_s=\frac{1}{s}\partial_k v^i_s+\frac{s-1}{s}\delta^i_k$,  by the definition of the norm $|\cdot|^*_{2+\epsilon,r,\eta}$, and the uniform bounds mentioned above, there exists a $C=C(g(\cdot),\gamma(\cdot),r')>0$ such that
$$| \partial \varphi |^*_{2+\epsilon,r,\eta} \leq \frac{C}{\eta^4}.$$ 

Hence $\hat g(s)=(\varphi_s)_*g(\bar t -r^2+s)$ satisfies
$$\sup_{\eta\in(0,1)}\eta^8 |\hat g|^*_{2+\epsilon,r,\eta}<C.$$

Now set, for $s\in[0,r^2]$,
$$\Omega_s=(\varphi_s\circ u)^{-1}(B(0,r)^+)$$ 
$$\phi_s=\varphi_s\circ u|_{\Omega_s}.$$
to construct the required parabolic coordinates $(\Omega_s,\phi_s)$ at $(\bar x,\bar t)$. Moreover, by the uniform bounds above choosing $r$ small enough we can achieve (\ref{prad1})-(\ref{prad3}).

\end{proof}
\begin{remark}
The quantity in (\ref{prad3}) in our case  is only there to deal with the lack of higher order compatibility at the boundary, while solving the harmonic map heat flow.  
\end{remark}

\begin{remark}
We may define parabolic coordinates around interior points in a similar way. Then, instead of the bound (\ref{prad3}) we would simply use the scale invariant $C^{2+\epsilon,\frac{2+\epsilon}{2}}$ norm. However, we will not need to do this, as Shi's local estimates provide higher order estimates of the curvature in the interior. 
\end{remark}

In analogy to the elliptic setting where one has the concept of the harmonic radius, introduced in \cite{and90}, we give the following definition.

\begin{definition}[Parabolic radius]
Consider a Ricci flow $(g(t),\gamma(t))$ on $M$ with $\Lambda$-controlled boundary and let $(\bar x,\bar t)\in \partial M\times (a,b]$. We define the parabolic radius $r^{Q}_p(\bar x,\bar t)$  to be the maximal $r>0$ for which there exist parabolic coordinates $(\Omega_s,\phi_s)$ of radius $r$ based at $(\bar x,\bar t)$  such that the bounds (\ref{prad1})-(\ref{prad3}) hold.
\end{definition}

\subsection{Curvature bounds.}
In the following lemma we show that on a Ricci flow with $\Lambda$-controlled boundary a lower bound of the parabolic radius yields bounds on derivatives of the ambient curvature $\Rm$ and the second fundamental form $\mathcal A$ of the boundary.

\begin{lemma}\label{bounds}
There exists an $\alpha=\alpha(Q)\in (0,1)$ such that if $(M,g(t),\gamma(t))$ is a Ricci flow with $\Lambda$-controlled boundary in $(a,b]$ and $(\bar x,\bar t)\in \partial M\times (a,b]$
then, for any $\eta\in(0,1)$ and $0\leq j\leq m-2$, the estimates
$$|\nabla^j \Rm(g(t))|_{g(t)}\leq C(Q,\Lambda,l,j,\eta)r_{p}^{-(j+2)}$$
and
$$|\nabla^{j+1}\mathcal A(g(t))|_{g(t)^T}\leq  C(Q,\Lambda,l,j,\eta)r_{p}^{-(j+2)},$$
hold for each $t\in  \left[\bar t -\left(\eta r_p\right)^2,\bar t\right]$ in $B_{g(t)}(\bar x,\alpha r_p)$, where $B_{g(t)}(\bar x,r)$ denotes a $g(t)$-distance ball in $M$ centered at $\bar x$ and $r_p=r^Q_{p}(\bar x,\bar t)$, the parabolic radius at $(\bar x,\bar t)$.
\end{lemma}

\begin{proof}
We will assume that $\bar t=0$ and supress the reference to $(\bar x,0)$ and $Q$.  Moreover, by rescaling we may also assume that $r_p=1$.

The proof is essentially an application of parabolic regularity theory. Since $r_p=1$, there are parabolic coordinates $(\Omega_s,\phi_s)$ based at $(\bar x,0)$ which define a flow $\hat g(s)=(\phi_s)_*g(s-1)$ on $B(0,1)^+$. Setting $\mathcal W(\hat g)^i=\hat g^{pq}\widehat \Gamma^i_{pq}$, $\hat g(s)$ satisfies the Ricci-DeTurck equation
$$\partial_s \hat g=-2\Ric(\hat g)+\mathcal L_{\mathcal W(\hat g)}\hat g,$$
in $B(0,1)^+$ and 
\begin{eqnarray}
\mathcal W(\hat g(s))&=&0,\nonumber\\
\mathcal H(\hat g(s))&=&\eta(u^1,\ldots,u^n,s),\nonumber\\
\hat g(s)^T-\frac{\textrm{tr}_{\gamma}\hat g(s)^T}{n}\gamma(s)&=&0,\nonumber
\end{eqnarray}
on $B(0,1)^+ \cap \{x^0=0\}$, $s\in (0,1]$.

By the assumption on the parabolic radius $\hat g$ satisfies the estimates (\ref{prad1})-(\ref{prad3}), for $r=1$. Moreover, in the coordinates $\phi_s|_{\partial M}$, $\gamma$ and $\mathcal H$ satisfy bounds
\begin{eqnarray}
\Lambda^{-2} Q^{-1}\delta\leq &\gamma& \leq \Lambda^2 Q \delta,\nonumber\\
|\gamma_{ij}|^*_{l,1}&\leq& C(Q,\Lambda),\nonumber\\
|\mathcal H|^*_{l-1,1} &\leq& C(Q,\Lambda).\nonumber
\end{eqnarray}

Thus, by parabolic regularity, for any $\eta\in(0,1)$ we obtain uniform $C^{l,l/2}$ estimates of $\hat g(s)$ on $B(0,\alpha')^+\times [1-\eta^2,1]$  for some $\alpha'(Q)\in (0,1)$. Namely,  we use cutoff functions to extend the Ricci DeTurck flow on $B(0,1)^+$  to a boundary value problem on the full $\mathbb R^{n+1}_+$, as is done in \cite{gianniotis}. This brings us exactly to the setting of \cite{MR0211083}. Note that the argument in \cite{gianniotis} requires  $\hat g(0)_{ij}=\delta_{ij}$ at $u=0$. However, we can achieve this by applying a linear transformation, which will be controlled in terms of $Q$, by (\ref{prad1}).

The higher order estimates of the curvature and the second fundamental form of $\hat g(s)$ on $B(0,\alpha')^+$ follow immediately. Namely, we obtain
$$|\widehat\nabla^j \Rm(\hat g(s))|_{\hat g(s)}\leq C(Q,\Lambda,l,j,\eta)$$
and
$$|\widehat\nabla^{j+1}\mathcal A(\hat g(s))|_{\hat g(s)^T}\leq  C(Q,\Lambda,l,j,\eta),$$
for $0\leq j\leq m-2$ and  $s\in [1-\eta^2,1]$. 

Now,  (\ref{prad1}) implies that $B_{\hat g(s)}(0,\alpha'Q^{-1/2})^+\subset B(0,\alpha')^+$. On the other hand, $B_{\hat g(s)}(0,\alpha'Q^{-1/2})^+$ is isometric to the $g(s-1)$-distance ball $B_{g(s-1)}(\bar x,\alpha'Q^{-1/2})$ via $\phi_s$, which means that for any $t\in [-\eta^2,0]$, $x\in B_{g(t)}(\bar x,\alpha'Q^{-1/2})$ and $x'\in B_{g(t)}(\bar x,\alpha'Q^{-1/2})\cap \partial M$ one has the estimates
\begin{eqnarray}
|\nabla^j \Rm( g(t))(x,t)|_{g(t)}&\leq& C(Q,\Lambda,l,j,\eta),\nonumber\\
|\nabla^{j+1}\mathcal A(g(t))(x',t)|_{g^T(t)}&\leq & C(Q,\Lambda,l,j,\eta).\nonumber
\end{eqnarray}

Setting $\alpha=\alpha'Q^{-1/2}$ we obtain the statement of the lemma.
\end{proof}

\subsection{Compactness.} Since we will need to deal with metrically incomplete manifolds in the following sections, we define below a notion of Cheeger-Gromov convergence suitable for this setting. Then, we prove a compactness theorem which will allow us to extract limits of sequences of incompete Ricci flows with boundary. 

\begin{definition}[Cheeger-Gromov convergence-weak form]\label{CGweak}
Let $(M_k,g_k(t),x_k)$, $(M_\infty,g_\infty(t), x_\infty)$ be, possibly incomplete, Ricci flows on pointed manifolds with boundary $M_k, M_\infty$, $t\in (a,b]$. Let $\gamma_k(t),\gamma_\infty(t)$ be one parameter families of metrics on $\partial M_k,\partial M_\infty$ such that $[g_k^T(t)]=[\gamma_k(t)]$ and  $[g_\infty^T(t)]=[\gamma_\infty(t)]$ for every $t\in (a,b]$.

We will say that $(M_k,g_k(t),\gamma_k(t),x_k)$ converge in the pointed $C^m$ Cheeger-Gromov sense to  \linebreak $(M_\infty,g_\infty(t),\gamma_\infty(t), x_\infty)$  if there exists an exhaustion $\{K_k\}$ of $M_\infty$ with compact sets and $C^{m+1}$ diffeomorphisms $F_k: K_k \rightarrow F_k(K_k)\subset M_k$ such that
\begin{enumerate}
\item $F_k(x_\infty)=x_k$,
\item $F_k|_{K_k\cap \partial M_\infty}:K_k\cap \partial M_\infty\rightarrow F_k(K_k\cap \partial M_\infty)\subset \partial M_k$ is a diffeomorphism,
\item $F_k^* g_k(t)\rightarrow g_\infty(t)$ smoothly and locally in $M_\infty\times (a,b]$ in $C^m$,
\item $F_k^* \gamma_k(t) \rightarrow \gamma_\infty(t)$ smoothly and locally in $\partial M_\infty\times (a,b]$ in $C^m$.
\end{enumerate}
\end{definition}
\begin{remark}
The $C^m$ topology above refers to objects with $q$ continuous derivatives in space and $\tau$ in time, for any $q,\tau$ with $2\tau + q\leq m$.
\end{remark}
\begin{notation}\label{converge}
We will adopt the notation $(M_k,g_k(t),\gamma_k(t),x_k)\rightharpoonup (M_\infty,g_\infty(t),\gamma_\infty(t),x_\infty)$ for such convergence.
\end{notation}

\begin{theorem}[Compactness theorem]\label{comp1}
Let $(M_k,g_k(t),\gamma_k(t),x_k)$ be a sequence of pointed Ricci flows with $\Lambda$-controlled boundary, $t\in (a,b]$. Suppose that for all $k$
\begin{enumerate}
\item $|\Rm(g_k(t))|\leq K$ in $M_k\times  (a,b]$, 
\item $i_{b,loc,g_k(0)}(x_k)\geq 4i_0$,   
%\item $\inj_{g_k^T(0)}(x_k)\geq i_0$,  
\item $r_{p,g_k}(x,t)\geq \zeta\left(\frac{D_{g_k}(x,t)}{\dist_{g_k^T(t)}(x,x_k)}\right)$  for all $(x,t)\in\partial M_k\times  (a,b]$,
\end{enumerate}
for some $K, i_0>0$ and a nondecreasing continuous function $\zeta$ with $\zeta(0)=0$, positive away from $0$. 

Then there is a Ricci flow on a  pointed manifold with boundary $(M_\infty,g_\infty(t), x_\infty)$ and a one parameter family of metrics $\gamma_\infty(t)$ on $\partial M_\infty$ such that
\begin{enumerate}
\item Up to subsequence, $(M_k,g_k(t),\gamma_k(t),x_k)\rightharpoonup (M_\infty,g_\infty(t),\gamma_\infty(t),x_\infty)$ in the $C^{m-3}$ topology.
\item If $i_{b,loc,g_k(0)}(x_k)\rightarrow \infty$ then $(M_\infty,g_\infty(t))$ is  complete and noncompact.
\end{enumerate}
\end{theorem}
To prove the compactness result, we will first need the following lemma. Suppose $(N,h)$ is a Riemannian manifold with boundary, possibly incomplete. Lemma \ref{intinj} essentially reduces the control of the interior injectivity radius of $N$ to that of its boundary, at least finite distance away from it (see also \cite{knox_compactness}).

For $0<r'<r<i_{b,loc,h}(x)$, denote by $A(x,r',r)$ the image of $B^\partial_h(x,r)\times \{s\nu, s\in [r',r)\}$ under the normal exponential map of the boundary, where $\nu$ is the inward pointing unit normal.

\begin{lemma}\label{intinj}
For any $y\in A(x,r',r)$ there exist positive constants $\kappa$ and $\rho_0$, depending on $i_{b,loc,h}(x)$, $r',r$, a lower bound on $vol(B^\partial_h(x, i_{b,loc,h}(x)))$, bounds of $|Rm(h)|$ in $A(x,0,i_{b,loc,h}(x))$ and $|\mathcal A(h)|$ on $B^\partial_h(x,i_{b,loc,h}(x))$,  such that
$$\frac{vol(B_h(y,\rho))}{\rho^{n+1}}\geq \kappa,$$
for all $\rho\leq \rho_0$.
\end{lemma}
\begin{proof}
It suffices to prove the statement for $\bar y= \expo_{x}(\frac{i_{b,loc,h(x)}}{2} \nu)$ since the full statement follows by volume comparison. 

In appropriate coordinates on $A(x,0,r)$ we have $h=dr^2+ h_r$, where $h_r$ are metrics on $B^\partial _h(x,i_{b,loc,h}(x))$, $r$ denoting the distance from $\partial N$. Setting $\bar h=dr^2+h_0$, the bounds for the curvature and the second fundamental form provide a uniform $C>1$ such that $C^{-1}\bar h\leq h\leq C \bar h$. Hence, there is a  $\rho_0=\rho_0(C)$ such that, for all $0<\rho\leq\rho_0$,  $B_{\bar y}(\rho)\subset A(x,0,r)$ and
$$\frac{vol_h(B_h(\bar y,\rho))}{\rho^{n+1}}\geq \kappa(C,v_0),$$
where $\frac{vol_{h^T}(B^\partial _h(x,\rho))}{\rho^n}\geq v_0$ for all $0<\rho<i_{b,loc,h}(x)$.
\end{proof}

\begin{proof}[Proof of Theorem \ref{comp1}]
The manifolds involved are incomplete and at this stage we only need the existence of a limit with the properties of the theorem. Thus we will make the assumption that each $M_k$ is just the image $C_k(x_k,4 i_0)$ of $B^\partial_{g_k(0)}(x_k,4 i_0)\times  \{s\nu, s\in [0,4 i_0)\}$ under the normal exponential map with respect to $g_k(0)$. Namely, we denote $C(x,r)=A(x,0,r)$.

Then, the assumption on the parabolic radius implies a uniform bound on the second fundamental form $\mathcal A(g_k(0))$ on $B^\partial_{g_k(0)}(x_k,3 i_0)$. This, together with  assumptions (1) and (3), controls $g_k(0)$ in $C_k(x_k,3 i_0)$ in the $C^0$ sense (as in the proof of Lemma \ref{intinj}) thus the sets $C_k(x_k,2i_0)$ remain a uniform distance away from $\overline{C_k(x_k,3 i_0)}\setminus C_k(x_k,3 i_0)^o$. 

To obtain higher order control in the interior, note that by Shi's estimates we obtain uniform bounds on derivatives of $\Rm(g_k(0))$ in $C_k(x_k,2i_0)$. Also, by Lemma \ref{intinj} and the $\Lambda$-controlled assumption there exist $v_0,c_0>0$ such that for each $y\in C_k(x_k,2i_0) $ and $r< c_0 \dist_{g_k(0)} (y, B^\partial_{g_k}(x_k, 2i_0))$ 
\begin{eqnarray}
 vol_{g_k(0)} (B_{g_k(0)}(y, r)) \geq v_0 r^{n+1}. 
\end{eqnarray}
%Also, the curvature bound and Lemma \ref{intinj} controls the injectivity radius of interior points in $ C_k(x_k,2i_0)$ in terms of their distance from $B^\partial_{g_k}(x_k, 2i_0)$. 
The volume ratio bound and the  bound on the curvature suffice for the blow-up argument in \cite{and90} to go through. Hence, the uniform bounds on the derivatives of the curvature and elliptic regularity imply that the $C^{m-1,\epsilon}$ harmonic radius of the interior points of $C_k(x_k,i_0)$ is also controlled, in terms of their $g_k(0)$-distance from $\partial M_k$.
%Now, we can use the results in \cite{and90} together with elliptic regularity to show that, the $C^{m-1,\epsilon}$ harmonic radius of the interior points of $C_k(x_k,i_0)$ is also controlled, in terms of their $g_k(0)$-distance from $\partial M_k$.

On the other hand, near the boundary, the work in \cite{AndTaylor} implies that the $C^{1,\epsilon}$ boundary harmonic radius of points on the boundary is also uniformly bounded below. Moreover, the assumption on the parabolic radius and Lemma \ref{bounds} provides bounds on derivatives of $\Rm(g_k(0))$ and $\mathcal A(g_k(0))$ up to the boundary. These bounds combined with Schauder estimates in the arguments of \cite{AndTaylor} control the $C^{m-1,\epsilon}$ harmonic radius of points on $C_k(x_k,i_0)\cap \partial M_k$.

%Regarding the boundary, our assumptions provide control of the curvature up to order $m-2$, via the Gauss equation. By \cite{and90}, it follows that in harmonic coordinates of uniform size the metrics $g_k^T(0)$ are controlled in $C^{m-1,\alpha}$. 

%By elliptic regularity (with the conformal class and mean curvature as boundary conditions) we obtain $C^{m-1,\alpha}$ control of $g_k(0)$ in the interior, in boundary harmonic coordinates and a smaller domain.

The discussion above shows that $\overline{C_k(x_k,i_0)}$ can be covered by harmonic coordinates of \linebreak controlled size in which $g_k(0)$ is $C^{m-1,\epsilon}$-controlled. Hence, standard arguments imply that \linebreak $(C_k(x_k,i_0), g_k(0) ,x_k)$ converge, up to subsequence, to an incomplete limit $(M_\infty,x_\infty,g_\infty(0))$  in the $C^{m-1}$ Cheeger-Gromov sense of Definition \ref{CGweak}.  For instance, see \cite{Kasue} and Section 2.6 of \cite{And89} for sequences of incomplete manifolds.

Then, using the curvature bounds of Lemma \ref{bounds} and Arzel\`{a}-Ascoli as in \cite{hamcomp}, we obtain that $F_k^*g_k(t)\rightarrow g_\infty(t)$ in the $C^{m-3}$ topology.

If $i_{b,loc,g_k(0)}(x_k)\rightarrow \infty$, take a sequence $r_i\rightarrow \infty$. The argument above shows that for each $i$, $(C_k(x_k,r_i),g_k(\cdot))$  converge as $k\rightarrow \infty$. Taking a diagonal sequence we obtain a noncompact complete limit. 

It remains to show that (4) of Definition \ref{CGweak} holds. First, recall that in the proof of Lemma \ref{bounds} we obtained higher order estimates for $\hat g$ in parabolic coordinates at any $(\bar x,\bar t)\in \partial M\times (a,b]$. These coordinates restrict to $\gamma(\bar t)$-harmonic coordinates around $\bar x$. Hence, since the flow of the DeTurck vector field fixes the boundary, we obtain $C^{l,l/2}$ control of $g^T(\cdot)$ in $\gamma(\bar t)$-harmonic coordinates of uniform size, in the time interval $(\bar t-r_p^2,\bar t]$. 

Note that these estimates hold regardless the time $\bar t$ we choose for the $\gamma(\bar t)$-harmonic coordinates, since $\gamma$ is controlled in $C^{l,l/2}$ (in time intervals $[c,d]\subset (a,b]$). Hence we have uniform $C^{m,\epsilon}$ control of $g^T(0)$ in $\gamma(0)$-harmonic coordinates.

By Schauder estimates and (1) of the Definition \ref{controlled_data}, the transition functions from $g^T(0)$-harmonic coordinates to $\gamma(0)$-harmonic coordinates are $C^{m+1,\epsilon}$-controlled. Therefore, $\gamma(0)$ is $C^{m,\epsilon}$-controlled in $g^T(0)$-harmonic coordinates as well.  Finally, by Definition \ref{controlled_data} we obtain $C^{l,l/2}$ control of $\gamma(\cdot)$ in the same coordinates.

Applying this to the sequence $(g_k(t),\gamma_k(t))$, and using Arzel\`{a}-Ascoli, we obtain a limit $\gamma_\infty(t)$ which satisfies $[\gamma_\infty(\cdot)]=[g^T_\infty(\cdot)]$. This finishes the proof of the theorem.
\end{proof}

\subsection{Lower semicontinuity of the parabolic radius.}

\begin{lemma}\label{lsemic}
The parabolic radius is lower semicontinuous with respect the Cheeger-Gromov topology in the sense that if $(M_k,g_k(t),\gamma_k(t),x_k)\rightharpoonup (M_\infty,g_\infty(t),\gamma_\infty(t),x_\infty)$ and $1<Q'<Q$ then
$$\liminf_k r^Q_{p,g_k}(x_k,\bar t)\geq r^{Q'}_{p,g_\infty}(x_\infty,\bar t)\quad\textrm{for all $\bar t\in (a,b]$.}$$
\end{lemma}
\begin{proof}
 To simplify notation, let $\bar t=0$ and denote $r_k=r^Q_{p,g_k}(x_k,0)$, $r_\infty=r^{Q'}_{p,g_\infty}(x_\infty,0)$. We will also assume that $g_k(t)$ and $\gamma_k(t)$ are defined on compact subsets $C\subset M_\infty$ and  $C\cap \partial M_\infty$, that $x_k\equiv x_\infty:=\bar x$, and  $g_k(t),\gamma_k(t)$ converge to $g_\infty(t),\gamma_\infty(t)$ smoothly and uniformly in compact subsets of $M_\infty$. 

Let $r_\sigma=r_\infty-\sigma$, for $\sigma>0$ small.  The goal is to construct, for $k$  large, parabolic coordinates $(\Omega_s^k,\phi^k_s)$ of radius $r_\sigma$ at $(\bar x,0)$ for the Ricci flow $(g_k(t),\gamma_k(t))$ such the bounds  (\ref{prad1})-(\ref{prad2}) hold.

Consider $(g_\infty(t),\gamma_\infty(t))$-parabolic coordinates $(\Omega_s,\phi_s)$ of radius $r_\infty$, $\phi_s:\Omega_s\rightarrow B(0,r_\infty)^+$, based at $(\bar x,0)$  in which the bounds (\ref{prad1})-(\ref{prad2}) hold for $Q'<Q$. 

Let $u=(u^0,\ldots,u^n)$ denote the coordinates in $\mathbb R^{n+1}$, where superscripts $1\leq i \leq n$ correspond to the directions along the boundary of the upper-half space $\mathbb R^{n+1}_+$. Note that by the definition of parabolic coordinates, 
\begin{equation}
 \Delta_{\gamma_\infty(0)} u^i =0,\label{gamma_infty_harmonic}
\end{equation}
where we write $\gamma_\infty$ for $(\phi_s)_*\gamma_\infty$, abusing slightly notation.

 Now, let $u_k^1,\ldots,u_k^n$ be solutions to the following Dirichlet problems in  $B(0,r_\sigma)\subset \mathbb R^n$
\begin{eqnarray}
\Delta_{\gamma_k(0)}u_k^i&=&0,\label{gamma_k_harmonic}\\
u_k^i|_{\partial B(0,r_\sigma)}&=&u^i.\nonumber
\end{eqnarray} 
By the convergence of $\gamma_k(0)$ to $\gamma_\infty(0)$ it follows that in the $u$-coordinates $|(\Delta_{\gamma_k}-\Delta_\infty)(f)|_{C^{1,\epsilon}}\rightarrow 0$, for any fixed function $f$ on $B(0,r_\sigma)$.
Moreover, by (\ref{gamma_infty_harmonic}), the functions $u_k^i-u^i$ satisfy 
\begin{eqnarray}
\Delta_{\gamma_k(0)}(u_k^i-u^i)&=&(\Delta_{\gamma_k(0)}-\Delta_{\gamma_\infty(0)})u^i,\nonumber\\
u_k^i-u^i|_{\partial B(0,r_\sigma)}&=&0.\nonumber
\end{eqnarray}
%$|u_k^i-u^i|_{C^{3,\epsilon}}\rightarrow 0$
Since $\gamma_k(0)$ are uniformly bounded in $C^{2,\epsilon}$, Schauder estimates imply that 
\begin{equation}
|u_k^i-u^i|_{C^{3,\epsilon}}\rightarrow 0.\label{conv1}
\end{equation}
In particular, $\{u_k^i\}_{i=1,\ldots,n}$ are $\gamma_k(0)$-harmonic coordinates, for large $k$.\\

Define the map $F_k:B(0,r_\infty)^+\rightarrow \mathbb R^{n+1}_+$ as
\begin{equation}
 F_k(u^0,u^1,\ldots,u^n)=(u^0,u_k^1(u^1,\ldots,u^n),\ldots,u_k^n(u^1,\ldots,u^n)).
\end{equation}
Clearly, for any small $\delta>0$ 
\begin{equation}
 B(0,r_\infty-\delta)^+\subset F_k(B(0,r_\infty)^+)\subset B(0,r_\infty+\delta)^+ , \label{squeeze}
\end{equation}
and $F_k$ is a diffeomorphism onto its image for large $k$, by (\ref{conv1}).\\

Now, let $\tau=r_\infty^2-r_\sigma^2$ and for any $k$ and $s\in [0,r_\sigma^2]$ define
\begin{equation}
 \Phi^k_s=F_k\circ \phi_{s+\tau}.
\end{equation}
It follows from (\ref{gamma_k_harmonic}) that the map $\Phi^k_s|_{\partial M_\infty}$ is $\gamma_k(0)$-harmonic.\\

By (\ref{squeeze}) it follows that for large $k$ we can set $\Omega^k_s=(\Phi^k_s)^{-1}(B(0,r_\sigma)^+)\subset \Omega_{s+\tau}$. We wish to construct $\phi^k_s:\Omega^k_s\rightarrow B(0,r_\sigma)^+$ satisfying
\begin{eqnarray}
\frac{\partial}{\partial s} \phi^k_s&=&\Delta_{g_k(s-r_\sigma^2),\delta} \phi^k_s,\label{h1}\\
\phi^k_0&=&\Phi^k_0,\label{h2}\\
\phi^k_s|_{\partial \Omega^k_s}&=&\Phi^k_s|_{\partial \Omega^k_s}.\label{h3}
\end{eqnarray}

In order to solve (\ref{h1}) we will use $\Phi^k_s$ to push-forward the problem to a domain $B_\sigma$ with smooth boundary, such that $B(0,r_\sigma)^+\subset B_\sigma\subset B(0,r_\infty)^+$. By (\ref{squeeze}), $B_\sigma\subset F_k(B(0,r_\infty)^+)$ (for large $k$), so this is possible. 

Set $h_{k,\sigma}(s)=(\Phi^k_s)_*(g_k(s-r_\sigma^2))$. Note that $\Phi^k_s$ is smooth even when $s=0$ and observe that to solve (\ref{h1})-(\ref{h3}) it suffices to solve the linear boundary value problem
\begin{eqnarray}
\frac{\partial}{\partial s}\chi_k&=&\Delta_{h_{k,\sigma}(s),\delta}\chi_k + D\chi_k(Y_k(s)),\\
\chi_k|_{s=0}&=&\ide_{B_\sigma},\\
\chi_k|_{\partial B(0,r_\sigma)^+}&=&\ide_{\partial B_\sigma},
\end{eqnarray}
where $Y_k(s)=D\Phi^k_s(\frac{\partial}{\partial s}(\Phi^k_s)^{-1})$ and $\chi_k(s)=\phi^k_s\circ(\Phi^k_s)^{-1}|_{B_\sigma}$. 

Now, let $h_s=(\Phi^k_s)_*(g_\infty(s-r_\sigma^2))$. The difference $\chi_k-u$ satisfies the linear problem
\begin{eqnarray}
\frac{\partial}{\partial s}(\chi_k-\ide)-\Delta_{h_{k,\sigma},\delta}(\chi_k-\ide)-D(\chi^k-\ide)(Y_k)&=&(\Delta_{h_{k,\sigma},\delta}-\Delta_{h_s,(F_k)_*\delta})u^i,\\
\chi_k-u|_{s=0}&=&0,\\
\chi_k-u|_{B_\sigma}&=&0.
\end{eqnarray}
This follows from
\begin{eqnarray}
\frac{\partial}{\partial s} \ide_{B_\sigma}&=&\frac{\partial}{\partial s} (\Phi^k_s\circ  (\Phi^k_s)^{-1}),\\
&=&(\frac{\partial}{\partial s} \Phi^k_s ) \circ  (\Phi^k_s)^{-1}+D\Phi^k_s (\frac{\partial}{\partial s}(\Phi^k_s)^{-1} ),\\
&=&\Delta_{h_s,(F_k)_*\delta} \ide_{B_\sigma}+ D\ide_{B_\sigma} (Y_k).
\end{eqnarray}
Since $h_{k,\sigma}\rightarrow h_s$ and $F_k\rightarrow \ide$ smoothly in $B_\sigma$, it follows from parabolic estimates that
\begin{eqnarray}
|\chi_k-u|_{C^{1+\epsilon,\frac{1+\epsilon}{2}}(B_\sigma\times [0,r_\sigma^2])}&\rightarrow& 0, \label{conv2}\\
|\chi_k-u|_{C^{3+\epsilon,\frac{3+\epsilon}{2}}(B_\sigma\times [(\eta r_\sigma)^2,r_\sigma^2])}&\rightarrow& 0, \label{conv3}
\end{eqnarray}
uniformly for any $\eta \in(0,1)$. In particular, $\phi^k_s$ are diffeomorphisms for large $k$.\\

Now, $\phi^k_s$ define parabolic coordinates of radius $r_{\sigma}$. Moreover, writing $\hat g^k(s)=(\phi^k_s)_* g_k(s)$, estimates (\ref{conv1}), (\ref{conv2}) and (\ref{conv3}) imply the bounds
\begin{eqnarray}
Q^{-1}\delta\leq \hat g^k(s)\leq Q\delta &&\qquad \textrm{ for $s\in [0,r_\sigma^2]$,}\nonumber\\
|\hat g^k_{ij}|^*_{\epsilon,r_\sigma}\leq Q,&&\nonumber\\
\sup_{\eta\in (0,1)} \eta^8 |\hat g^k_{ij}|^*_{2+\epsilon,r_\sigma,\eta}\leq Q,&&\nonumber
\end{eqnarray}
for large $k$, since $Q'<Q$. Hence $\liminf r_k\geq r_\sigma=r_{\infty} -\sigma $ for every $\sigma>0$.
\end{proof}

\subsection{Controlling the parabolic radius.}
In the following we show that certain geometric bounds suffice to control the parabolic radius from below. The proof is via a blow up argument, essentially on the same line as in \cite{and90}, \cite{AndTaylor} and \cite{knox_compactness}.

\begin{lemma}\label{prad_control}
Let $(g(t),\gamma(t))$ be a Ricci flow on $M$ with $\Lambda$-controlled boundary, $t\in (a,b]$, and suppose there exist $K,i_0,D>0$ such that
\begin{enumerate}
\item $|\Rm(g(t))|\leq K$ in $M\times (a,b]$,
\item $|\mathcal A(g(t))|\leq K$ on $\partial M\times (a,b]$,
\item $i_{b,loc,g(t)}(x)\geq i_0 D_g(x,t)$  for all $(x,t)\in \partial M\times  (a,b]$,
%\item $\diam(\partial M, g^T(t))\leq D$.$\inj_{g^T(t)}(x)\geq i_0 D_g(x,t)$
\end{enumerate}
Then, if $\Lambda>0$ is as in Definition \ref{controlled_data}, there exists $c=c(K,i_0,D,\Lambda,l)>0$ such that 
$$\frac{r^Q_p(x,t)}{D_g(x,t)}\geq c. $$
\end{lemma}
\begin{proof}
It suffices to show that given any precompact domain $V\subset \partial M$ and $[a',b]\subset  (a,b]$ the estimate 
$$\frac{r^Q_p(x,t)}{D'_g(x,t)}\geq  c(K,i_0,D,\Lambda,l)>0$$
holds in $V\times (a',b]$, where $D'_{g,V}(x,t)=\min\{\dist_{g^T(t)}(x, \bar V\setminus V), (t-a')^{1/2},1\}$, assuming that condition (3) holds in $V$ with $D_g$ replaced by $D'_g$.

Suppose there is a sequence of counterexamples, namely manifolds with boundary $M_k$, Ricci flows $(g_k(t),\gamma_k(t))$ with $\Lambda$-controlled boundary satisfying bounds (1)-(3),  $V_k\times [a'_k,b]\subset \partial M_k\times (a,b]$ and spacetime points $(y_k,t_k)\in  V_k \times  (a'_k,b]$ such that
\begin{equation}
 \frac{r^Q_{p,g_k}(y_k,t_k)}{D'_{g_k,V_k}(y_k,t_k)}=\epsilon_k\rightarrow  0.
\end{equation}

Since $r^Q_{p,g_k}>0$ in $\bar V_k\times [a'_k,b]$ for any $k$, by Lemma \ref{lsemic} we may assume that for $1<Q'<Q$
\begin{equation}
 \frac{r^Q_{p,g_k}}{D'_{g_k,V_k}}(x,t)\geq \frac{r^{Q'}_{p,g_k}}{D'_{g_k,V_k}}(y_k,t_k), \textrm{ for all } (x,t)\in V_k\times (a'_k,b].
\end{equation}
 Moreover, $r_k:=r^{Q'}_{p,g_k}(y_k,t_k)\rightarrow 0$. 

Consider the pointed sequence $(M_k,y_k)$  and the rescaled flows $(h_k(t),\bar \gamma_k(t))$,  where $$h_k(t)=r_k^{-2}g_k(t_k+t r_k^2),$$
$$\bar \gamma_k(t)=r_k^{-2}\gamma_k(t_k+t r_k^2).$$
These rescaled flows are defined in $t\in (-\frac{t_k-a'_k}{r_k^2}, \frac{b-t_k}{r_k^2})$ and also have $\Lambda$-controlled boundary. In addition
\begin{itemize}
\item[i.] $|\Rm(h_k(t))|\rightarrow 0$ in $M_k\times(-\frac{t_k-a'_k}{r_k^2}, \frac{b-t_k}{r_k^2})$,
\item[ii.] $|\mathcal A(h_k(t))| \rightarrow 0$ on $\partial M\times (-\frac{t_k-a'_k}{r_k^2}, \frac{b-t_k}{r_k^2})$,
\item[iii.] $r^{Q'}_{p,h_k}(y_k,0)=1$, since the parabolic radius scales like distance,
\item[iv.] $vol_{h_k^T(0)}(B^\partial_{h_k^T}(y_k,r)) \geq v_0 r^n$ for every $0<r< \dist_{g_k^T(t_k)}(y_k,\bar V\setminus V ) r_k^{-1}\rightarrow \infty$,%$\inj_{V_k, h_k^T(0)}(y_k)\geq \frac{i_0}{\epsilon_k} \rightarrow \infty$,
\item[v.] $i_{b,loc,h_k(0)}(y_k)\geq \frac{i_0}{\epsilon_k}\rightarrow \infty$,
\item[vi.] $\frac{t_k-a'_k}{r_k^2}\rightarrow \infty$. 
\end{itemize}
Moreover, we have $r^Q_{p,h_k}(x,t)\geq \frac{D'_{h_k}(x,t)}{D'_{h_k}(y_k,0)}r^{Q'}_{p,h_k}(y_k,0)$ which gives 
\begin{equation}
 r^Q_{p,h_k}(x,t)\geq c(\dist_{h_k^T(0)}(x,y_k))>0,
\end{equation}
uniformly for large $k$ on $V_k\times [-1,0]$.

By Theorem \ref{comp1} the rescaled flows in $V_k\times (-1,0]$ have a Cheeger-Gromov limit,  a pointed manifold with boundary $(M_\infty, y_\infty)$  with a complete Ricci  flow $h_\infty(t)$. Moreover, $(M_\infty,h_\infty)$ is flat and $\partial M_\infty$ is totally geodesic (and flat, from the Gauss equation). By (iv) we conclude that $\partial M_\infty$ is isometric to $(\mathbb R^n,\delta)$. 

Now, by (i), (ii) and (v) above and comparison geometry we have that $M_\infty$ is isometric to $(\mathbb R^{n+1},\delta)$. Namely, the second fundamental form of the level sets of the distance functions from the boundary will converge to zero, uniformly in fixed distance from the boundary. The claim follows, as these level sets converge smoothly to the corresponding level sets in the limit.

This implies that $r^{Q''}_{p,\delta}(0,0)=\infty$, for any $1<Q''<Q'$, which contradicts Lemma \ref{lsemic} and the fact that $r^{Q'}_{p,h_k}(y_k,0)=1$.
\end{proof}

\subsection{Boundary estimates.}
We finish this section by putting together Lemmata \ref{bounds} and \ref{prad_control}  to obtain local higher order estimates up to the boundary for the Ricci flow.

\begin{theorem}\label{estimates}
Let $(M,g(t),\gamma(t),q)$, $t\in [0,T]$, be a pointed, possibly incomplete, Ricci flow with $\Lambda$-controlled boundary in $(0,T]$. Suppose
\begin{enumerate}
\item $|\Rm(g(t))|_{g(t)}\leq K$ in $M$ and $|\mathcal A(g(t))|_{g^T(t)}\leq K$ on $\partial M$ for all $t\in [0,T]$.
\item $i_{b,loc,g(t)}(q)\geq i_0$, for all $t\in [0,T]$.
%\item $\inj_{g^T(0)}(q)\geq i_0$.
\end{enumerate}
For any $j=1,\ldots,m-2$ and $\tau>0$, there exists a constant $C=C(n,\tau,T,\Lambda,l,j,K,i_0)>0$ such that for $t\in [\tau,T]$ 
\begin{eqnarray}
|\nabla^j \Rm(g(t))|_{g(t)}&\leq& C, \quad\textrm{in $C_{g(0)}(q,i_0/2)$.}\label{horderests}\\
|\nabla^{j+1} \mathcal A(g(t))|_{g^T(t)}&\leq& C, \quad\textrm{on $C_{g(0)}(q,i_0/2)\cap \partial M$},\nonumber
\end{eqnarray}
where $C_g(q,r)$ denotes the image under the normal exponential map of $B^\partial_{g^T}(q,r)\times \{s\nu,\;s\in[0,r)\}$.
\end{theorem}
\begin{proof}
Let $x\in B^\partial_{g(0)}(q,r_0)$ with $r_0<i_0$. From the definition of $i_{b,loc,g(t)}(q)$ it follows that $i_{b,loc,g(t)}(x)\geq i_0-r_0$, for all $t\in [0,T]$. 
%\begin{itemize}
%\item[(a)] From the definition of $i_{b,loc,g(t)}(q)$ it follows that $i_{b,loc,g(t)}(x)\geq i_0-r_0$, for all $t\in [0,T]$. 
%\item[(b)] The Gauss equation gives a bound on the curvature of $(\partial M,g^T(t))$. Moreover, the assumption of $\Lambda$-controlled boundary in $(0,T]$ implies that there is a $c=c(Q,\Lambda,r_0)$ such that
%\begin{eqnarray}
% \frac{vol_{g(t)}(B^\partial_{g(t)}(x,\rho))}{\rho^n}\geq c, \textrm{for any } \rho\leq i_0-r_0.\nonumber
%\end{eqnarray}
%It then follows from \cite{CGT} that $\inj_{g^T(t)}(x) \geq c(i_0,r_0,K)$. 
%\\This follows from \cite{CGT},  and the fact that all the metrics $g^T(t)$ are comparable gives a uniform lower bound on the volume ratio of balls centered at $x$.
%\end{itemize}
By Lemma \ref{prad_control}, the parabolic radius is bounded below uniformly in intervals $[\tau,T]$ in the ball $B^\partial_{g(0)}(q,i_0/2)$. Therefore, there is a $\delta$-neighbourhood of  $B^\partial_{g(0)}(q,i_0/2)$ in $M$ (with respect to the $g(0)$ metric) where estimates (\ref{horderests}) hold for some $C$, by Lemma \ref{bounds}. 

This proves that the required estimates hold on a neighbourhood of $C(q,i_0/2)\cap\partial M$. Shi's interior estimates then handle the higher derivatives of the curvature in $C(q,i_0/2)$ away from the boundary. 
\end{proof}

Theorem \ref{main_result} is the global version of the result above.
\begin{proof}[Proof of Theorem \ref{main_result}]
The proof is essentially the same as that of Theorem \ref{estimates}. We only need to show that we can estimate on $i_{b,g(t)}$ in terms of $K$ and $i_0$.

Let $\nu(x,t)$ be the inward pointing unit normal to $\partial M$ and $\exp^t$ the exponential map with respect to the metric $g(t)$. We consider the maximal geodesics $\gamma_{x,t}(s)=\exp^t_x(s \nu(x,t))$ and their $g(t)$-length $ L_t(\gamma_{x,t}) $.

Observe that there exists a $c(K)>0$ such that these geodesics don't have focal points for $s< c(K)$. Since $\partial M$ is compact, it follows that whenever $i_{b,g(t)} < c(K)$ 
$$i_{b,g(t)}=\frac{1}{2}\min\left\{L_t(\gamma_{x,t}), \; x\in \partial M\right\}.$$
In particular, there exists an $x_t\in \partial M$ such that $i_{b,g(t)}=\frac{1}{2} L_t(\gamma_{x_t,t})$. By the first variation formula of length $\gamma_{x_t,t}$ is $g(t)$-perpendicular to $\partial M$.

We define $$\frac{d}{dt} i_{b,g(t)}=\liminf_{h\rightarrow 0^+} \frac{ i_{b,g(t)}- i_{b,g(t-h)}}{h}.$$
Since $\gamma_{x_{t},t}$ are geodesics perpendicular to $\partial M$, we obtain that
\begin{eqnarray} 
\left.\frac{d}{dt}\right|_{t=t_0} i_{b,g(t)}&\geq& \left.\frac{1}{2}\frac{d}{dt}\right|_{t=t_0} L_t(\gamma_{x_{t_0},t})=\left.\frac{1}{2}\frac{d}{dt}\right|_{t=t_0} L_t(\gamma_{x_{t_0},t_0})\label{ineq}\\
&\geq& C(n,K)L_{t_0}(\gamma_{x_0,t_0})=C(n,K)i_{b,g(t_0)}, \nonumber
\end{eqnarray}
where the last inequality  follows from the curvature bound and the Ricci flow equation.

Therefore, using (\ref{ineq}) we obtain control of the boundary injectivity radius for $t>0$.
\end{proof}

\section{A compactness theorem.}\label{comp}

In this section we prove a version of Hamilton's compactness theorem for sequences of Ricci flows on manifolds with boundary. Below we define the notion of convergence we use, along the lines of \cite{AndTaylor}.

\begin{definition}[Cheeger-Gromov convergence-strong form]\label{CGstrong}
Let $(M_k,g_k(t),x_k)$, $(M_\infty,g_\infty(t), x_\infty)$ be complete Ricci flows on pointed manifolds with boundary $M_k, M_\infty$, $t\in (a,b]$. Let $\gamma_k(t),\gamma_\infty(t)$ be one parameter families of metrics on $\partial M_k,\partial M_\infty$ such that $[g_k^T(t)]=[\gamma_k(t)]$ and  $[g_\infty^T(t)]=[\gamma_\infty(t)]$.

We will say that $(M_k,g_k(t),\gamma_k(t),x_k)$ converge in the strong pointed $C^{m}$ Cheeger-Gromov sense to  $(M_\infty,g_\infty(t),\gamma_\infty(t), x_\infty)$  if there exists a sequence $R_k\rightarrow +\infty$, an exhaustion $\{K_k\}$ of $M_\infty$ by compact sets such that $B_{g_\infty(0)}(x_\infty,R_k)\subset K_k$ and $C^{m+1}$ diffeomorphisms $F_k: K_k \rightarrow F_k(K_k)\subset M_k$ such that
\begin{enumerate}
\item $F_k(x_\infty)=x_k$.
\item $B_{g_k(0)}(x_k,R_k)\subset F_k(K_k)$.
\item $F_k|_{K_k\cap \partial M_\infty}:K_k\cap \partial M_\infty\rightarrow F_k(K_k\cap \partial M_\infty)\subset \partial M_k$ is a diffeomorphism.
\item $F_k^* g_k(t)\rightarrow g_\infty(t)$ smoothly and locally in $M_\infty\times (a,b]$ in $C^{m}$.
\item $F_k^* \gamma_k(t) \rightarrow \gamma_\infty(t)$ smoothly and locally in $\partial M_\infty\times (a,b]$ in $C^{m}$.
\end{enumerate}
For such convergence we write
$$(M_k,g_k(t),\gamma_k(t), p_k)\rightarrow (M_\infty,g_\infty(t),\gamma_\infty(t),p_\infty).$$
\end{definition}
\begin{remark}
Unlike Definition \ref{CGweak}, Cheeger-Gromov limits in the strong sense defined above are complete and unique.
\end{remark}

\begin{theorem}\label{comp2}
Let $(M_k,p_k)$ be a pointed sequence of manifolds with compact boundary, and $(g_k(t),\gamma_k(t))$ be complete Ricci flows on $M_k$, $t\in (a,b]$ with $\Lambda$-controlled  boundary in $(a,b]$. Assume
\begin{enumerate}
\item $|\Rm(g_k)|_{g_k}\leq K$ in $M_k\times (a,b]$.
\item $|\mathcal A(g_k)|_{g_k^T}\leq K$ in $\partial M_k\times (a,b]$.
%\item $\inj_{g^T_k(0)}(x)\geq i_0$, for all $x\in\partial M_k$.
\item $i_{b,g_k(0)}\geq i_0$.
\end{enumerate}
for all $k$. Then there is a pointed manifold with boundary $(M_\infty, p_\infty)$,  a Ricci flow $g_\infty(t)$ on $M_k$ and a family of metrics $\gamma_\infty(t)$ on $\partial M_\infty$ such that, up to subsequence, 
$$(M_k,g_k(t),\gamma_k(t), p_k)\rightarrow (M_\infty,g_\infty(t),\gamma_\infty(t),p_\infty),$$
in the $C^{m-3}$ topology.
\end{theorem}
\begin{proof}
We sketch the proof as it is essentially similar to the proof of Theorem \ref{comp1} (following \cite{hamcomp}). First, as in the proofs of Theorems \ref{main_result} and \ref{estimates}, the assumptions of the theorem provide a uniform lower bound on $i_{b,g(t)}$ and the injectivity radius of the boundary. 

Then,  Theorems \ref{bounds}, \ref{prad_control} and Shi's local derivative of curvature estimates give
\begin{enumerate}
\item $|\nabla^j \Rm(g_k(0))|_{g_k(0)}\leq C,$
\item $|\nabla^{j+1}\mathcal A(g_k(0))|_{g^T_{k}(0)}\leq C,$
\end{enumerate}
for some $C$ independent of $k$ (and appropriate order depending on $m$).

Using the compactness result in \cite{AndTaylor} and elliptic regularity one obtains a pointed smooth Cheeger-Gromov limit $(M_\infty,g_\infty(0))$. Here, the interior injectivity radius control follows from \cite{CGT},  the curvature bound and the volume bound of Lemma \ref{intinj} .

Now, as in \cite{hamcomp}, the curvature bound and Arzel\`{a}-Ascoli show that a subsequence converges to a limit  flow $g_\infty(t)$ in the sense of Definition \ref{CGstrong}. The convergence of $\gamma_k(t)$ follows like in Theorem \ref{comp1}.
\end{proof}

\section{Estimates on 3-manifolds.}\label{threemnflds}
Now we assume that $M$ is a three dimensional manifold with compact boundary. The following theorem shows that along a three dimensional Ricci flow the second fundamental form of the boundary is essentially controlled by the ambient curvature. In the proof we use Liouville's Theorem for bounded subharmonic functions in $\mathbb R^2$. It is unknown yet whether Theorem \ref{sffb} holds in higher dimensions.

\begin{theorem}\label{sffb}
Let $(g(t),\gamma(t))$, $t\in [0,T]$, be a three dimensional complete Ricci flow  with compact $\Lambda$-controlled boundary in $[0,T]$. Assume
\begin{eqnarray}
|\Rm(g)|_{g}&\leq& K \quad \textrm{in }  M\times [0,T],\nonumber\\
i_{b,g(0)}&\geq& \frac{i_0}{\|\mathcal A(g(0))\|_\infty}.\nonumber
\end{eqnarray}
Then, for $0\leq j\leq m-4$, there exist $C_j=C(K,T,\Lambda,l,i_0,j)>0$ such that the second fundamental form $\mathcal A$ of $\partial M$ and the boundary injectivity radius satisfy
\begin{eqnarray}
 i_{b,g(t)}&\geq& C_0^{-1}  \quad\textrm{for all }  t\in[0,T], \label{ibestimate}\\
| \mathcal A(g)|_{g^T}&\leq& C_0 \quad\textrm{in }  \partial M\times [0,T],\label{sffestimate1}\\
| \nabla^j \mathcal A(g(t))|_{g^{T}(t)}&\leq& \frac{C_j}{t^{\frac{j+1}{2}}}  \quad\textrm{in }  \partial M\times (0,T], \textrm{ for }\; j\geq 1,\label{sffestimate2}\\
| \nabla^{j-1} \Rm(g(t))|_{g(t)}&\leq& \frac{C_j}{t^{\frac{j+1}{2}}}  \quad\textrm{in }   M\times (0,T], \textrm{ for }\; j\geq 2. \label{dcurvestimate}
\end{eqnarray}
\end{theorem}
\begin{proof}
We begin with some necessary non-collapsing estimates. First, we obtain an estimate on $i_{b,g(t)}$ for $t>0$. From comparison geometry, if $\lambda_{max} (t)$ is the largest eigenvalue of $\mathcal A(g(t))$ on $\partial M$ and
 $\mathbf K(t)=\max\{\sqrt K,\lambda_{max}(t)\} $ then 
$$i_{b,g(t)}\geq \min\left\{ \frac{\pi}{2\mathbf K(t)}, \frac{1}{2}\min\left\{L_t(\gamma_{x,t}), \; x\in \partial M\right\}\right\},$$
where $L_t(\gamma_{x,t})$ is as in the proof of Theorem \ref{main_result}. Moreover,  when $i_{b,g(t)}< \frac{\pi}{2\mathbf K(t)}$ we have
\begin{equation}\frac{d}{dt} i_{b,g(t)}\geq -C i_{b,g(t)},\label{ibeqn}\end{equation}
for some positive constant $C$ depending only on $K$ and $n$, by (\ref{ineq}).

Let $\mathbf K_{max}=\max_{t\in [0,T]} \mathbf K(t)$. It follows from (\ref{ibeqn}) that there exists an $0<\alpha=\alpha(K,i_0,T)<1$ with the property that 
\begin{equation}
i_{b,g(t)}\geq \alpha  \frac{\pi}{2\mathbf K_{max}},\label{ibbound}
\end{equation}
for all $t\in [0,T]$.\\

Second, note that the $\Lambda$-control of the boundary implies a volume ratio bound for the induced metrics $g^T$ on $\partial M$. Namely, there exist $r_0,v>0$ such that for every $p\in \partial M$
\begin{equation}
 \frac{vol(B_{g^T(t)}(p,r))}{r^n}\geq v, \quad\textrm{for every } r\leq r_0. \label{vol_bound}
\end{equation}

Now, consider a sequence of counterexamples, namely $\Lambda$-controlled Ricci flows $(M_k,g_k(t),\gamma_k(t))$ and $(p_k,t_k)\in \partial M_k\times [0,T]$ such that
$$|\mathcal A(g_k(t_k))(p_k)|=\max_{\partial M_k\times [0,T]}|\mathcal A(g_k(t))(p)|\rightarrow \infty.$$
Set $ A_k=|\mathcal A(g(t_k))(p_k)|$, and consider the pointed manifolds with boundary $(M_k,p_k)$ with the rescaled metrics 
\begin{eqnarray}
h_k&=& A_k^2 g_k(t_k),\nonumber\\
\bar \gamma_k&=& A_k^2 \gamma_k(t_k).\nonumber
\end{eqnarray}

Along this sequence, 
\begin{eqnarray}
 \|\Rm(h_k)\|_\infty&\rightarrow& 0, \label{curv_to_zero}\\
\|\mathcal H(h_k)\|_\infty&\rightarrow& 0, \label{mcurv_to_zero}\\
\|\mathcal A(h_k)\|_\infty&\leq& 1, \\
|\mathcal A(h_k)(p_k)|&=&1,\\
\inj_{h^T_k}&\geq& \delta_0,\;\textrm{by \cite{CGT} and (\ref{vol_bound})},\\
i_{b,h_k}&\geq& \delta_0, \;\textrm{by (\ref{ibbound})}.
\end{eqnarray} 
In addition, the interior injectivity radius is controlled from Lemma \ref{intinj} and, since the boundary is $\Lambda$-controlled, $\mathcal H(h_k)$ is controlled in the Lipschitz sense. It follows by the compactness result in \cite{AndTaylor} that there is a subsequence converging in the $C^{1,\epsilon}$ topology to a $C^{1,\epsilon}$ Cheeger-Gromov limit $(M_\infty, h_\infty,p_\infty)$ satisfying $|\mathcal A_\infty(p_\infty)|=1$.\\

Moreover, $h_k^T=e^{2U_k}\bar \gamma_k(t_k)$ for appropriate functions $U_k$ on $\partial M_k$ which satisfy the elliptic equation
\begin{equation}
-2\Delta_{h^T_k} U_k= R_{h_k^T}-R_{\bar \gamma_k}e^{-2U_k}.\label{conformal_eqn}
\end{equation}
By assumption $|U_k|\leq\ln\Lambda$, and  $h^T_k$ is controlled in $C^{1,\epsilon}$ in $h^T_k$-harmonic coordinates. Elliptic regularity shows that $U_k$ is also controlled in $C^{1,\epsilon}$ (in a slightly smaller domain).

Therefore, after possibly passing to a subsequence, we may assume that $F_k^*\bar \gamma_k\rightarrow \bar\gamma_\infty$ in $C^{1,\epsilon}$, $F_k$ being the diffeomorphisms associated to the Cheeger-Gromov convergence of $h_k$ to $h_\infty$ (see Definition \ref{CGstrong}). It is clear that $\bar \gamma_\infty$ is just the Euclidean metric $\delta$ in $\mathbb R^2$, and that there exists a function $U_\infty$ so that $h^T_\infty=e^{2U_\infty}\delta$. Moreover, $F_k^*U_k\rightarrow U_\infty$ uniformly locally  in $C^{1,\epsilon}$.\\

From (\ref{conformal_eqn}), it follows that for $\psi \in C^{\infty}_c(\mathbb R^2)$, 
\begin{equation}
 \int_{\mathbb R^2}\frac{1}{2}\left( e^{2U_k}R_{h_k^T} - R_{\bar\gamma_k}\right)\psi dvol_{\bar \gamma_k}=\int_{\mathbb R^2}\langle \nabla U_k,\nabla\psi\rangle dvol_{\bar \gamma_k}\rightarrow \int_{\mathbb R^2}\langle \nabla U_\infty,\nabla \psi\rangle dvol_\delta.
\end{equation}
Now, the Gauss equation together with (\ref{curv_to_zero}) and (\ref{mcurv_to_zero}) gives $R_{h_k^T}\leq \varepsilon_k$, for some positive $\varepsilon_k \rightarrow 0$. Therefore, for $\psi\geq 0$ 
\begin{equation}
 \int_{\mathbb R^2}\frac{1}{2}\left( e^{2U_k}R_{h_k^T} - R_{\bar\gamma_k}\right)\psi dvol_{\bar \gamma_k}\leq \varepsilon'_k\rightarrow 0,
\end{equation}
hence $U_\infty$ is subharmonic in the weak sense. Morever, since $|U_\infty|\leq \ln \Lambda$ it follows from Liouville's theorem that it is constant.\\

Now, multiplying the Gauss equation for $h_k$ by $\psi \in C^{\infty}_c(\mathbb R^2)$, using (\ref{conformal_eqn}) and integrating we obtain
\begin{eqnarray}
 \int_{\mathbb R^2}|\mathcal A(h_k)|^2\psi dvol_{h_k^T}&=&-2\int_{\mathbb R^2}\langle \nabla U_k,\nabla \psi\rangle dvol_{h_k^T}\label{Azero}\\
&&-\int_{\mathbb R^2} (R_{\bar\gamma_k}e^{-2U_k}-\mathcal H(h_k)^2 +R_{h_k}-2Ric_{h_k}(N_k,N_k))\psi dvol_{h_k^T},\nonumber
\end{eqnarray}
where $N_k$ denotes the outward pointing unit normal with respect to $h_k$. Here, to simplify notation we write $h_k$ instead of $F_k^* h_k$.

It is clear that the right-hand side of (\ref{Azero}) converges to zero, which implies that $\mathcal A(h_\infty)\equiv 0$.  On the other hand, $|\mathcal A_\infty(p_\infty)|=1$ by the $C^{1,\epsilon}$ convergence, which is a contradiction.

This shows that $|\mathcal A(g(t))|$ is bounded above in terms of $K$, $i_0$ and $n$, proving (\ref{sffestimate1}). Then, (\ref{ibestimate}) follows from (\ref{ibbound}).\\

We describe the proof of estimate (\ref{sffestimate2}) briefly. Consider a sequence of counterexamples $(M_k,g_k(t),\gamma_k(t))$ such that there exist $(p_k,t_k)\in \partial M_k\times (0,T]$
so that

$$t_k^{\frac{j+1}{2}}|\nabla^ j \mathcal A(g_k(t_k))(p_k)|=\max_{\partial M_k\times [0,T]} t^{\frac{j+1}{2}}|\nabla^j \mathcal A(g_k(t))(p)|\rightarrow \infty.$$

Under the assumptions of the theorem and estimate (\ref{sffestimate1}), setting $Q_k= | \nabla^ j \mathcal A(g_k(t_k))(p_k)|^{\frac{2}{j+1}}$ and rescaling
\begin{eqnarray}
h_k(t)&=& Q_k g_k(t_k+tQ_k^{-1}),\nonumber\\
\bar \gamma_k(t)&=&Q_k \gamma_k(t_k+t Q_k^{-1}),\nonumber
\end{eqnarray}
we may apply Theorem \ref{comp2} to obtain 
$$(M_k,h_k(0),\bar \gamma_k(0),p_k)\rightarrow (M_\infty,h_\infty,\bar \gamma_\infty,p_\infty).$$
The limit will have a totally geodesic boundary, hence $|\nabla^j \mathcal A_\infty|=0$. This contradicts that $|\nabla^j \mathcal A_\infty(p_\infty)|=1$, which holds if $m$ is large enough.\\

Estimate (\ref{dcurvestimate}) is shown similarly.
\end{proof}

\begin{remark}
Clearly estimate (\ref{sffestimate1}) holds for any Riemannian 3-manifold with compact boundary satisfying the assumptions of Theorem \ref{sffb}, a fact which may be of independent interest. It suffices that the boundary is $\Lambda$-controlled up to order $1+\epsilon$. 
\end{remark}

Theorem \ref{sffb} allows us to improve the continuation principle in \cite{gianniotis} and Theorem \ref{comp2} in dimension three.
\begin{corollary}\label{cp}
Let $g(t)$, $t\in [0,T)$, $T<\infty$,  be a maximal Ricci flow on a compact 3-manifold with boundary $M$. Suppose that there exist smooth data $\gamma(t)$ and $\eta(t)$ defined for $t\in[0,T')$, $T'>T$, such that
\begin{eqnarray}
[g^T(t)]&=&[\gamma(t)],\nonumber\\
\mathcal H(g(t))&=&\eta.\nonumber
\end{eqnarray}
for all $t\in [0,T)$. Then, 
$$\sup_{M\times [0,T)}|\Rm(g(t)|_{g(t)}=\infty.$$
\end{corollary}
\begin{proof}
By \cite{gianniotis} 
$$\sup_{M\times [0,T)}|\Rm(g(t)|_{g(t)}+\sup_{\partial M\times [0,T)}|\mathcal A(g(t))|_{g^T(t)}=\infty.$$
In addition, since the data $\gamma, \eta$ are assumed to be smooth up to time $T'>T$, condition (2) in the Definition \ref{controlled_data}  is easily seen to be satisfied. Moreover, the bound on the curvature implies that $g(t)$ is uniformly equivalent to $g(0)$ for all $t\in [0,T)$, which suffices for condition (1) of Definition \ref{controlled_data} to hold. Hence, $(g(t),\gamma(t))$ is $\Lambda$-controlled in $[0,T]$ and Theorem \ref{sffb} asserts that the second fundamental form remains bounded as $t\rightarrow T$, which finishes the proof.
\end{proof}
\begin{corollary}
If $M$ is a compact 3-manifold, Theorem \ref{comp2} holds without assumption (2) on the second fundamental form.
\end{corollary}

\begin{acknowledgements}
The author would like to thank Michael Anderson and Peter Topping for many interesting discussions. This research has been supported by the EPSRC on a Programme Grant entitled ``Singularities of Geometric Partial Differential Equations'' reference number EP/K00865X/1.
\end{acknowledgements}

\def\cprime{$'$} \def\cprime{$'$} \def\cprime{$'$}
\providecommand{\bysame}{\leavevmode\hbox to3em{\hrulefill}\thinspace}
\providecommand{\MR}{\relax\ifhmode\unskip\space\fi MR }
\providecommand{\MRhref}[2]{%
  \href{http://www.ams.org/mathscinet-getitem?mr=#1}{#2}
}
\providecommand{\href}[2]{#2}

\end{document}